\documentclass{IEEEtran4PSCC}
\ifCLASSINFOpdf
   \usepackage[pdftex]{graphicx}
\else
   \usepackage[dvips]{graphicx}
\fi
\usepackage[cmex10]{amsmath}
\usepackage{algorithm}
\usepackage{algorithmic}
\ifCLASSOPTIONcompsoc
 \usepackage[caption=false,font=normalsize,labelfont=sf,textfont=sf]{subfig}
\else
 \usepackage[caption=false,font=footnotesize]{subfig}
\fi
\usepackage[backend=biber, style=ieee]{biblatex}
\usepackage{multicol}
\usepackage{amsthm}
\newtheorem{theorem}{Theorem}

\newtheorem{proposition}{Proposition}

\makeatletter
\let\old@ps@headings\ps@headings
\let\old@ps@IEEEtitlepagestyle\ps@IEEEtitlepagestyle
\def\psccfooter#1{%
    \def\ps@headings{%
        \old@ps@headings%
        \def\@oddfoot{\strut\hfill#1\hfill\strut}%
        \def\@evenfoot{\strut\hfill#1\hfill\strut}%
    }%
    \def\ps@IEEEtitlepagestyle{%
        \old@ps@IEEEtitlepagestyle%
        \def\@oddfoot{\strut\hfill#1\hfill\strut}%
        \def\@evenfoot{\strut\hfill#1\hfill\strut}%
    }%
    \ps@headings%
}
\makeatother

\begin{document}

%
\title{Feedback Enhancement of Time Series Aggregation for Power System Expansion Planning}

\author{\IEEEauthorblockN{Ruiqi Zhang,
Ensieh Sharifnia,
Simon H. Tindemans}
\IEEEauthorblockA{Department of Electrical Sustainable Energy\\
Delft University of Technology,
Delft, The Netherlands\\ \{r.zhang-3, e.sharifnia, s.h.tindemans\}@tudelft.nl}
}

\maketitle
\raggedbottom

\begin{abstract}
As a consequence of the high variability of load demand and renewable generation, long-term and high-resolution inputs are required for power system expansion planning, making the problem intractable in real-world applications. Time series aggregation (TSA), which captures representative patterns, reduces temporal complexity while providing similar planning outputs. However, purely statistical clustering, even when enhanced with predefined ``extremes'', can overlook system-specific critical operating conditions, making it unreliable across real-world systems. Therefore, this paper links TSA accuracy on specific system operation and final solution quality, which becomes a practical bound with mean-based TSA approaches. It is observed that the distribution of operational errors is highly imbalanced, such that a few representatives dominate the total error. This paper proposes an adaptive clustering strategy based on feedback enhancement of TSA that iteratively identifies poor-performing representatives with high operational error and re-clusters only their associated periods. A study shows that the feedback enhancement improves the decision error and tighten the bound significantly compared with the plain mean-based clustering method, offering a diagnostic for TSA quality, while balancing the computational effort with solution accuracy.
\end{abstract}

\begin{IEEEkeywords}
Expansion planning; Numerical approximation; Operational cost; Optimization; Time series aggregation
\end{IEEEkeywords}

\thanksto{\noindent This work is part of the project Multi-Energy System Smart Linking (MuESSLi), funded by the Collaborative Research for Energy SYstem Modelling (CRESYM) non-profit association.}

\section{Introduction}
The rise in electrified demand together with large shares of renewable energy sources (RES) has made power system expansion planning both more critical and more challenging. Capturing variability patterns in load and RES typically requires high-resolution time series spanning a full year (e.g., 8760 hourly steps) or more, making full-resolution planning models computationally intractable for large-scale systems \cite{teichgraeber_time-series_2022}. To address this computational burden while preserving essential temporal characteristics, time series aggregation (TSA) is widely applied to construct a much smaller set of representative periods that approximate the complete time series. The resulting planning model with representative periods serves as a surrogate for the full-scale model, which is far cheaper to solve while aiming to generate similar planning outcomes \cite{li_representative_2022}.

Assuming that statistically similar periods would yield similar planning outcomes, early TSA practices applied simple heuristics such as seasonal averaging and day/night block aggregation \cite{short_regional_2011, mallapragada_impact_2018}. More recently, clustering and optimization-based TSA methods (e.g., k-means \cite{green_divide_2014}, k-medoids \cite{tejada-arango_enhanced_2018}, hierarchical clustering \cite{moradi-sepahvand_representative_2024}, direct optimization formulations \cite{gonzato_long_2021}, and their combinations \cite{riyahi_optimizing_2025}) treat elementary time segments (hours/days/weeks) as high-dimensional points, group them based on dissimilarity measures (typically Euclidean distance), and select cluster centroids or medoids as their representatives \cite{teichgraeber_clustering_2019, pinel_clustering_2020}. Some studies also incorporate global features, such as total load duration curves, to better aggregate system characteristics \cite{poncelet_selecting_2017}. 

However, aggregation inevitably loses information from the original time series, and more critically, the mapping from time series features to planning and operational costs is system-specific and nonlinear \cite{teichgraeber_time-series_2022}. Consequently, increasing statistical similarity does not guarantee improved planning accuracy, leading to inconsistent performance across systems \cite{wogrin_time_2023}. This fundamental limitation suggests that TSA should be informed by how the power system actually operates rather than solely by time series similarity \cite{hilbers_importance_2019, zhang_model-adaptive_2023}.

Several approaches include expert knowledge by hard-coding ``extreme periods'' as representatives, including peak load \cite{pinel_clustering_2020}, extreme RES availability \cite{garcia-cerezo_priority_2022}, extreme net load \cite{yeganefar_improvement_2020,moradi-sepahvand_capturing_2023}, and peak ramps \cite{scott_clustering_2019}. While often helpful in test cases, these heuristics are chosen in a one-size-fits-all manner from the profiles themselves. Their effectiveness cannot be guaranteed universally \cite{wogrin_time_2023}, as truly ``extreme'' periods depend on network topology, operational constraints, and cost parameters that remain unknown until the planning model is solved \cite{scott_clustering_2019}.

To incorporate system operation directly in TSA, recent research has explored objective-oriented TSA methods. Instead of clustering based on the raw time series profiles, these methods group periods by the similarity of investment decisions \cite{zhang_model-adaptive_2023} or its costs \cite{sun_data-driven_2019} made on single-period planning optimizations. Hilbers \emph{et al.} ran full-scale optimizations with fixed investments, used each period’s operational cost as its ``importance'', and then performed clustering the high generation cost periods with more representatives \cite{hilbers_reducing_2023}.

Furthermore, monitoring the operational cost difference between representative and full time series (for a fixed investment decision), Bahl \emph{et al.} \cite{bahl_typical_2018} iteratively add representative periods until reaching a threshold, though without mathematical guarantees. Teichgraeber and Brandt \cite{teichgraeber_clustering_2019} proved that for linear programming (LP) models with certain structure, the operational cost error bounds the objective value difference between optimal and estimated decisions. This result was extended to mixed-integer linear programming (MILP) with binary investment variables \cite{li_representative_2022, santosuosso_optimal_2025}. However, existing methods still rely on plain clustering with more representatives or heuristic selection of extreme periods based on absolute operational cost, rather than targeting the operational error that ultimately determines decision quality.

This paper directly addresses the operational cost error introduced by TSA, analyzing its distribution and proposing a feedback enhancement framework that guides clustering toward poorly performing clusters. Firstly, one cost-oriented relationship is derived linking TSA accuracy in system operation to final decision quality, which becomes a practical bound for mean-based aggregation with convex operational subproblems. Secondly, looking into the system operation error introduced by TSA, our analysis reveals that representatives perform highly imbalanced across the full-scale time series, with a small subset dominating the total error. This imbalance can not be resolved by simply increasing representative count. To address this, one system-aware re-clustering mechanism is proposed that identifies representatives with the largest operational errors and selectively improves only their associated time periods, improving both operational accuracy and decision quality. Finally, through a test case study on a modified IEEE 24-bus system, we demonstrate that the proposed feedback mechanism significantly outperforms conventional TSA methods in both operational cost estimation and planning objective accuracy.
\section{Error Analysis of Representative Periods in Power System Expansion Planning}
\label{sec:Error Bounds of Reduced Problems}
\subsection{Expansion Planning with Representative Periods}
We consider the power system expansion planning problem that selects optimal investment decisions $x\in \mathcal{X}$ based on the minimization of total (expected) cost $C^{\mathrm{tot}}\left(x, \mathcal{D} \right)$, consisting of investment cost, $C^{\mathrm{invest}}(x)$, and operational cost, $C^{\mathrm{op}}(x,\mathcal{D})$:
\begin{align}
    \label{eq:opt_full_planning_tot}
    &\arg\min_{x\in \mathcal{X}} \quad C^{\mathrm{tot}}\left(x, \mathcal{D} \right) = C^{\mathrm{invest}}(x) + C^{\mathrm{op}}(x,\mathcal{D})
\end{align}
Here, $\mathcal{D}$ represents the set of operational scenarios to be considered. Specifically, we describe these scenarios in the form of (potentially ordered) time series segments with weight $w_d$ and feature vector $T_d$ (e.g., 24-hour daily load and RES capacity factors), forming the scenario set $\mathcal{D}=\{(w_d,T_d)\}_{d=1}^{|\mathcal{D}|}$. In this paper, the segment length is assumed to be one day, but other lengths can also be used. With such formatted time series inputs, the operational cost of the power system with investments $x$ is taken as:
\begin{subequations}
    \allowdisplaybreaks
    \label{eq:opt_full_planning_op}
    \begin{align}
        \label{eq:opt_full_planning_op_obj}
        & C^{\mathrm{op}}(x,\mathcal{D})=\min_{\{y_d\}_{d=1}^{|\mathcal{D}|}} \sum\nolimits_{d=1}^{|\mathcal{D}|} w_d\cdot C^{\mathrm{op}}(y_d;x,T_d)\\
        & \mbox{s.t.} \underbrace{y_d\in \mathcal{Y}_{x,d},}_{\mbox{intra-period constraints}} \forall d\in \{1,\ldots,|\mathcal{D}|\}, \\
        & \underbrace{\left(y_d,y_{d+1}\right)\in \mathcal{Y}_{x,d,d+1},}_{\mbox{inter-period constraints}} \quad \forall d\in \{1,\ldots,\left|\mathcal{D}\right|-1\}.
    \end{align}
\end{subequations}
Here, $y_d$ are operational variables and $\mathcal{Y}_{x,d}$ and $\mathcal{Y}_{x,d,d+1}$ represent the intra-period and inter-period constraints (e.g., generator ramping, storage, line limits, and power balance).

To sufficiently capture the patterns of load, solar, and wind variation, the scenario set should reflect future operational scenarios well. We assume such a set is available and refer to it as the full-scale scenario set $\mathcal{D}^{\mathrm{full}}$ (e.g., 365 daily profiles with weights all one representing one year). Solving the full-scale expansion planning model \eqref{eq:opt_full_planning_tot} with $\mathcal{D}^{\mathrm{full}}$ gives the \emph{reference decision} $x^\star$ with (optimal) cost $C^{\mathrm{tot}}\left(x^\star, \mathcal{D}^{\mathrm{full}} \right)$. However, direct calculation of $x^\star$ is often intractable in practical applications.

Instead, TSA approaches solve problem \eqref{eq:opt_full_planning_tot}-\eqref{eq:opt_full_planning_op} using a reduced scenario set $\hat{\mathcal{D}} = \{(\hat{w}_k,\hat{T}_k)\}_{k=1}^{|\hat{\mathcal{D}}|}$ with $|\hat{\mathcal{D}}| \ll |\mathcal{D}|$. The reduced scenario should satisfy $\sum_d w_d = \sum_k \hat{w}_k$ and be designed such that the minimizer $\hat{x}$ of $C^{\mathrm{tot}}\left(\hat{x}, \hat{\mathcal{D}}\right)$ (i.e., the \emph{approximate decision}) approximates the reference decision $x^\star$.

This formulation lies at the basis of common representative day frameworks \cite{moradi-sepahvand_representative_2024}, as well as chronological time-period clustering (CTPC) (with hourly periods) \cite{pineda_chronological_2018, garcia-cerezo_priority_2022}. Often, representative periods are determined based on clustering, which partitions the set $\mathcal{D}$ into $K$ clusters $\bigcup_{k=1}^{K}\mathcal{D}_k=\mathcal{D}^{\mathrm{full}}$ and determines representative periods $\hat{T}_k$ by computing the mean or medoid of the scenario parameters in $\mathcal{D}_k$. We consider representative days (RDs) throughout this paper.

\subsection{Performance Metrics of Representative Periods}
To evaluate how well the TSA based reduced-scale problem approximates the full-scale problem in the cost domain, we distinguish three metrics: (1) simplification error, (2) decision error, and (3) operational estimation error.

\subsubsection{Simplification error (reduced vs. full optimum)}
The simplification error is the difference in total cost between the computed optimal cost of the full and reduced planning models:
\begin{align}
    \label{eq:Def_SimplificationError}
    \Delta C^{\mathrm{obj}}(\hat{x}) \equiv C^{\mathrm{tot}}\left(\hat{x},\hat{\mathcal{D}}\right) - C^{\mathrm{tot}}\left(x^\star,\mathcal{D}^{\mathrm{full}}\right).
\end{align}
It is a common metric for evaluating the quality of the  reduced-scale solution $\hat{x}$ \cite{alvarez_novel_2017, tejada-arango_enhanced_2018, pineda_chronological_2018, pinel_clustering_2020}. However, because the operational cost for $\hat{x}$ is estimated based on RDs, this error is difficult to interpret. In cases where the operational cost is under-estimated, a falsely low total cost for $\hat{x}$ might occur, resulting in a misleadingly small error when compared to the optimal total cost of $x^\star$.

\subsubsection{Decision error (true suboptimality of the estimated decision)}
Additionally, we quantify the sub-optimality of a given investment option $x$ by evaluating its actual operational cost $C^{\mathrm{tot}}\left(x, \mathcal{D}^{\mathrm{full}}\right)$ using the full scenario set \cite{li_representative_2022, moradi-sepahvand_representative_2024, anderson_representative_2024}. The difference with the cost of the reference solution, $C^{\mathrm{tot}}(x^\star,\mathcal{D}^{\mathrm{full}})$, is the \emph{decision error}, given by
\begin{align}
    \label{eq:Def_DecisionError}
    \Delta C^{\mathrm{tot}}\left(x\right) \equiv C^{\mathrm{tot}}\left(x,\mathcal{D}^{\mathrm{full}}\right) - C^{\mathrm{tot}}\left(x^\star,\mathcal{D}^{\mathrm{full}}\right).
\end{align}
$\Delta C^{\mathrm{tot}}\left(x\right)$ is non-negative by definition and directly measures the suboptimality of a given decision. In real-world cases, the reference solution $x^\star$ is not available, so \eqref{eq:Def_DecisionError} cannot be evaluated. However, the full-scale operational cost, $C^{\mathrm{tot}}\left(\hat{x}, \mathcal{D}^{\mathrm{full}}\right)$, provides an upper bound of the optimal total cost, serving as an important reference for the decision makers' consideration \cite{li_representative_2022, teichgraeber_clustering_2019, santosuosso_optimal_2025}.

\subsubsection{Operational estimation error}
The final metric, the \emph{operational estimation error}, reflects how well RDs can represent actual power system operation. For a given investment decision $x$, it is defined as \cite{bahl_time-series_2017, bahl_typical_2018}:
\begin{align}
    \label{eq:Def_OperationalCostError}
    \nonumber
    \Delta C^{\mathrm{op}}\left(x\right) & \equiv C^{\mathrm{op}}\left(x, \hat{\mathcal{D}}\right) - C^{\mathrm{op}}\left(x, \mathcal{D}^{\mathrm{full}}\right) \\
    & = C^{\mathrm{tot}}\left(x, \hat{\mathcal{D}}\right) - C^{\mathrm{tot}}\left(x, \mathcal{D}^{\mathrm{full}}\right)
\end{align}
Its sign indicates whether the operational cost is under-estimated (-) or over-estimated (+) compared to the actual operational cost. 

Moreover, we see that for the investment decision $\hat{x}$ of the reduced problem, the three errors are related to each other as 
\begin{equation}
    \Delta C^{\mathrm{tot}}\left(\hat{x}\right) = \Delta C^{\mathrm{obj}}(\hat{x}) - \Delta C^{\mathrm{op}}\left(\hat{x}\right).
\end{equation}
This relation is shown illustrated in Fig.~\ref{fig:relationship}.

\begin{figure}
    \centering
    \includegraphics[width=0.99\linewidth]{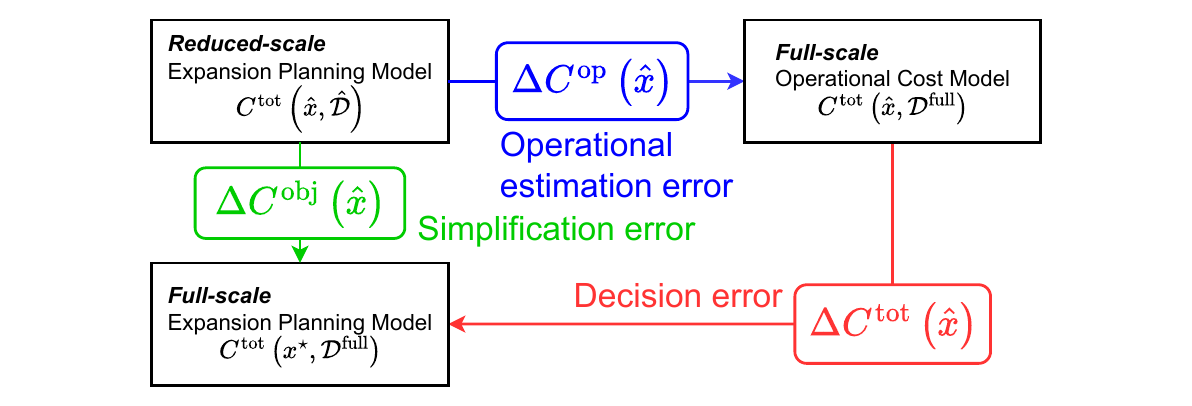}
    \caption{Relationship between three optimization models: reduced-scale planning model; full-scale operational cost model; and the reference full-scale planning model.}
    \label{fig:relationship}
\end{figure}

\subsection{Linking Operational Performance to Planning Outcomes}
Based on the optimization nature of both the reduced and full-scale problems, one bounding of decision error and operational estimation error is proposed, linking the performance of RDs and power system operation.
\begin{proposition}[General Bound on Decision Error]
    \label{proposition:general_bound_decision_error}
    For any reduced set $\hat{\mathcal{D}}$ and corresponding optimal decision $\hat{x}$, its decision error is non-negative and upper-bounded by the difference of operational estimation errors for the actual optimal decision $x^\star$ and the estimated decision $\hat{x}$.
    \begin{align}
    \label{eq:proposition_general_bound_decision_error}
    0 \leq \Delta C^{\mathrm{tot}}\left(\hat{x}\right) \leq \Delta C^{\mathrm{op}}\left(x^\star\right) - \Delta C^{\mathrm{op}}\left(\hat{x}\right)
\end{align}
\end{proposition}
\begin{proof}
    The first inequality, $0 \le \Delta C^{\mathrm{tot}}\left(\hat{x}\right)$ follows directly from \eqref{eq:Def_DecisionError} and the fact that $x^\star$ is the cost-minimizing solution. 
    
    The optimality of $\hat{x}$ on $\hat{\mathcal{D}}$ implies, for all $x\in \mathcal{X}$:
    \begin{align}
        C^{\mathrm{invest}}(\hat{x}) + C^{\mathrm{op}}(\hat{x},\hat{\mathcal{D}})
        \leq C^{\mathrm{invest}}(x) + C^{\mathrm{op}}(x,\hat{\mathcal{D}})
    \end{align}
    Adding and subtracting $C^{\mathrm{op}}(\hat{x},\mathcal{D}^{\mathrm{full}})$ on the left side and $C^{\mathrm{op}}(x,\mathcal{D}^{\mathrm{full}})$ on the right side gives:
    \begin{align}
        C^{\mathrm{\mathrm{tot}}}(\hat{x},\mathcal{D}^{\mathrm{full}}) + \Delta C^{\mathrm{op}}(\hat{x})
        \leq C^{\mathrm{\mathrm{tot}}}(x,\mathcal{D}^{\mathrm{full}}) + \Delta C^{\mathrm{op}}(x)
    \end{align}
    Setting $x=x^\star$ and reordering completes the proof.
\end{proof}
Proposition~\ref{proposition:general_bound_decision_error} implies that the performance of RDs in estimating the full-scale decision is highly relevant to its capacity of estimating the actual operational cost, regardless the TSA approaches applied and how the operational problem is formulated with candidate investments.

\subsection{Practical Error Bound for Mean-based Aggregation}
To compute the bounding value in Proposition~\ref{proposition:general_bound_decision_error}, knowledge of $x^\star$ is required, which is not available in practice. Nevertheless, the two components of the bound reflect how well the representative days capture system operation under different configurations. For certain formulations of the planning problem, mean-based clustering methods exhibit a systematic behavior to underestimate operational costs across different power systems \cite{teichgraeber_clustering_2019, li_representative_2022}.

Prioritizing ease of solution over accuracy, operational problems with linear constraints and convex cost functions are widely applied in energy system planning problems \cite{li_representative_2022, moradi-sepahvand_representative_2024, teichgraeber_clustering_2019, garcia-cerezo_priority_2022}. If we neglect inter-period constraints and consider models where time series data only loaded affects the right hand side of constraints, the planning problem is structured as \eqref{eq:opt_full_planning_tot} with the linear operational subproblem 
\begin{subequations}
    \label{eq:opt_RD_planning_op_bounded}
    \begin{align}
        \label{eq:opt_RD_planning_op_bounded_objective}
        & C^{\mathrm{op}}(x,\mathcal{D})=\min_{\{y_d\}_{d=1}^{|\mathcal{D}|}} \sum\nolimits_{d=1}^{|\mathcal{D}|} w_d\cdot C^{\mathrm{op}}(y_d;x), \\
        \label{eq:opt_RD_planning_op_bounded_constraint_leq}
        & \mbox{s.t.} \quad A_xy_d\leq M_x\cdot T_d + b_x, \quad \forall d\in \mathcal{D}, \\
        \label{eq:opt_RD_planning_op_bounded_constraint_eq}
        & \quad \quad B_xy_d = N_x\cdot T_d + d_x, \quad \forall d\in \mathcal{D},
    \end{align}
\end{subequations}
with scenario-dependent operational decisions $\{y_d\}_{d=1}^{|\mathcal{D}|}$ and cost functions $C^{\mathrm{op}}(y_d;x)$ that are convex in $y_d$. We note that the overarching investment problem (to determine $x$) does not need to be linear or convex.

In the following, we consider mean-based clustering, where representatives are cluster averages. Slightly abusing notation, we refer to scenario indices in clusters as $d \in \mathcal{D}_k$, and define
\begin{align}
    \label{eq:mean-based clustering}
    \hat{\mathcal{D}}=\{(\hat{w}_k,\hat{T}_k)\}_{k=1}^{K}=\left\{\left(\sum_{d \in \mathcal{D}_k} w_d,\frac{\sum_{d\in \mathcal{D}_k} w_d T_d}{\sum_{d \in \mathcal{D}_k} w_d}\right)\right\}_{k=1}^{K}
\end{align}
\begin{proposition}[Under-Estimation with Mean-based Clustering \cite{teichgraeber_clustering_2019, li_representative_2022}]
    \label{proposition:Under-estimation}
    Consider planning problems formulated as \eqref{eq:opt_full_planning_tot} with operational cost as \eqref{eq:opt_RD_planning_op_bounded}, and mean-based scenario clustering \eqref{eq:mean-based clustering}. Then, for any $x\in \mathcal{X}$, the reduced-scale model underestimates the full-scale operation.
    \begin{align} \label{eq:proposition-under-estimation}
    \Delta C^{\mathrm{op}}\left(x\right) \equiv C^{\mathrm{tot}}\left(x, \hat{\mathcal{D}}\right) - C^{\mathrm{tot}}\left(x, \mathcal{D}^{\mathrm{full}}\right) \leq 0
\end{align}
\end{proposition}
\begin{proof}
    \allowdisplaybreaks
    Following \cite{li_representative_2022}, for a fixed investment decision $x$, suppose $\{y_d\}_{d=1}^{|\mathcal{D}|}$ is a feasible solution for the full-scale planning problem. Within each cluster $\mathcal{D}_k$, we average the operational constraints \eqref{eq:opt_RD_planning_op_bounded_constraint_leq} and \eqref{eq:opt_RD_planning_op_bounded_constraint_eq} with weights $w_d$, resulting in:
    \begin{align}
        A_x\sum_{d\in \mathcal{D}_k}\frac{w_d y_d}{\hat{w}_k} &\leq M_x\sum_{d\in \mathcal{D}_k} \frac{w_d T_d}{\hat{w}_k} + b_x, && \forall k=1,\ldots, |\hat{\mathcal{D}}|, \\
        C_x\sum_{d\in \mathcal{D}_k}\frac{w_d y_d}{\hat{w}_k} &= N_x \sum_{d\in \mathcal{D}_k} \frac{w_d T_d}{\hat{w}_k} + d_x, && \forall k=1,\ldots, |\hat{\mathcal{D}}|.
    \end{align}
    These are exactly the operational constraints for the reduced-scale problem. Therefore, $\left\{\hat{y}_k=\sum_{d\in \mathcal{D}_k}\frac{w_d y_d}{\hat{w}_k}\right\}_{k=1}^{|\hat{\mathcal{D}}|}$ is also a feasible solution set for the reduced planning model. In addition, as $C^{\mathrm{op}}(y_d;x)$ is a convex function, we have:
    \begin{align}
    \nonumber
        C^{\mathrm{op}}(\hat{y}_k;x)&=C^{\mathrm{op}}\left(\sum_{d\in \mathcal{D}_k}\frac{w_d y_d}{\hat{w}_k}; x\right)\\
        &\leq \frac{1}{\hat{w}_k}\sum_{d\in\mathcal{D}_k} w_d  C^{\mathrm{op}}(y_d;x).
    \end{align}
    Therefore, aggregating all clusters, we have:
    \begin{align}
        \sum\nolimits_{k=1}^{|\hat{\mathcal{D}}|} \hat{w}_k C^{\mathrm{op}}(\hat{y}_k;x) \leq \sum\nolimits_{d=1}^{|\mathcal{D}|} w_d C^{\mathrm{op}}(y_d;x)
    \end{align}
    Therefore, if a feasible solution of the full-scale problem exists, it is always possible to find a feasible solution of the reduced-scale problem that yields the same or lower objective value. This proves \eqref{eq:proposition-under-estimation}, because
    \begin{align}
        C^{\mathrm{op}}(x,\hat{\mathcal{D}}) \leq C^{\mathrm{op}}(x,\mathcal{D}^{\mathrm{full}}).
    \end{align}
\end{proof}
\begin{theorem}
    \label{theorem:Bounding_OptimalityGap}
    For planning problems \eqref{eq:opt_full_planning_tot} with operational models formulated as \eqref{eq:opt_RD_planning_op_bounded} and mean-based scenario clustering
    \eqref{eq:mean-based clustering}, the total cost error of the reduced problem is non-negative and upper bounded by the negative operational estimation error of the estimated decision:
    \begin{align}
        0 \leq \Delta C^{\mathrm{tot}}\left(\hat{x}\right) \leq -\Delta C^{\mathrm{op}}\left(\hat{x}\right)
    \end{align}
\end{theorem}
\begin{proof}
    Given Propositions \ref{proposition:general_bound_decision_error} and \ref{proposition:Under-estimation}, we have:
    \begin{align}
        \label{eq:Def_Bounding_OperationalCostError}
        \Delta C^{\mathrm{op}}\left(\hat{x}\right) \leq \Delta C^{\mathrm{op}}\left(x^\star\right)\leq 0
    \end{align}
    Therefore, Proposition \ref{proposition:general_bound_decision_error} still holds after removing $\Delta C^{\mathrm{op}}(x^\star)$, completing the proof.
\end{proof}

In conclusion, regardless the clustering approaches, it is proved that the performance of RDs in estimating the full-scale planning problems is highly relevant to its capacity of estimating the actual operational cost. Especially when the mean points are selected as representatives, a practical bound can be derived for one generally applied planning framework. Moreover, a targeted reduction of the operational estimation error $\Delta C^{\mathrm{op}}\left(\hat{x}\right)$ is guaranteed to tighten the bound on the decision error.
\section{Case Study: IEEE 24-bus RTE Test System}
\label{sec:Case Study}
To evaluate the impacts of clustering errors on the final cost using clustering methods, a case study of generation and transmission co-expansion planning model is implemented on the modified IEEE 24-bus RTE test system with load and wind patterns of the Netherlands in 2019 (also used in \cite{moradi-sepahvand_representative_2024}). This system includes 22 loads, 10 thermal units, 6 candidate wind investments, and 5 candidate transmission lines. The planning problem minimizes the total investment and operational costs, subject to generator and line capacity limits, ramping constraints of thermal units, load shedding, wind curtailment, and power balance based on the DC optimal power flow formulation.

The full-scale time series of load demand and available wind capacity patterns span a full year at hourly resolution. These were normalized into load and wind capacity factors, and subsequently segmented into 365 daily time-series vectors. Clustering was performed to find RDs on the joint daily factor features $(\text{load},\ \text{wind})$. Specifically, mean-based hierarchical clustering (Ward’s linkage) was used with the dissimilarity metric between clusters $\mathcal{D}_a$ and $\mathcal{D}_b$:
\begin{align}
    \text{Dist}(\mathcal{D}_a,\mathcal{D}_b)=\frac{2\lvert \mathcal{D}_a\rvert \lvert \mathcal{D}_b\rvert}{\lvert \mathcal{D}_a\rvert+\lvert \mathcal{D}_b\rvert}\left|\left|\frac{1}{\lvert \mathcal{D}_a\rvert}\sum_{i\in \mathcal{D}_a}T_i-\frac{1}{\lvert \mathcal{D}_b\rvert}\sum_{j\in \mathcal{D}_b}T_j\right|\right|^2
\end{align}
Here $\mathcal{D}_a$ and $\mathcal{D}_b$ denote two clusters of days, and $d_i$ is the feature vector (load and wind factor) of day $i$. The resulting RD set was determined as the mean value of each cluster according to \eqref{eq:mean-based clustering}.

The expansion planning model was formulated using Pyomo in Python and solved with Gurobi and computations were executed on an Apple M1 Pro with 16 GB RAM. All code and data used in this study are available \cite{zhang_code_2025}.

\section{Analysis of Operational Estimation Error}
\label{sec:Analysis of Operational Estimation Error}
It is evident that the operational estimation error is a crucial metric to evaluate and improve the performance of RDs. Therefore, looking into the operational estimation error, its characteristics across both original and RDs are analyzed.

\subsection{Operational Estimation Error over Original Days}
In reduced-scale planning models, the operational costs of individual original days are approximated using the cluster representatives. By comparing the actual and estimated operational costs for each original day, the operational estimation error per original day can be obtained; the sum of these errors over all days constitutes the total operational estimation error. To investigate also the statistical dissimilarity between the original days and its RDs, the Euclidean distance between an original day $T_i$ and its assigned RD $\hat{T}_j$,
\begin{align}
    \text{Dist}(T_i,\hat{T}_j) = \left|\left|T_i-\hat{T}_j\right|\right|_2,
\end{align}
was computed and reported as the time series error.

Fig.~\ref{fig:comparison_oc_ts} shows the distribution of operational estimation error and time series error across original days for 20 and 40 RDs. It is observed that the operational estimation error is extremely unevenly distributed. A small fraction of days shows extremely large errors, contributing disproportionately to the total error, whereas the majority are well-represented with relatively low errors. Increasing the number of RDs reduces the total error but does not improve the unbalanced distribution: on some days the error is significantly higher than on other days.
\begin{figure}[bhtp]
    \centering
        \subfloat{\includegraphics[width=0.9\columnwidth]{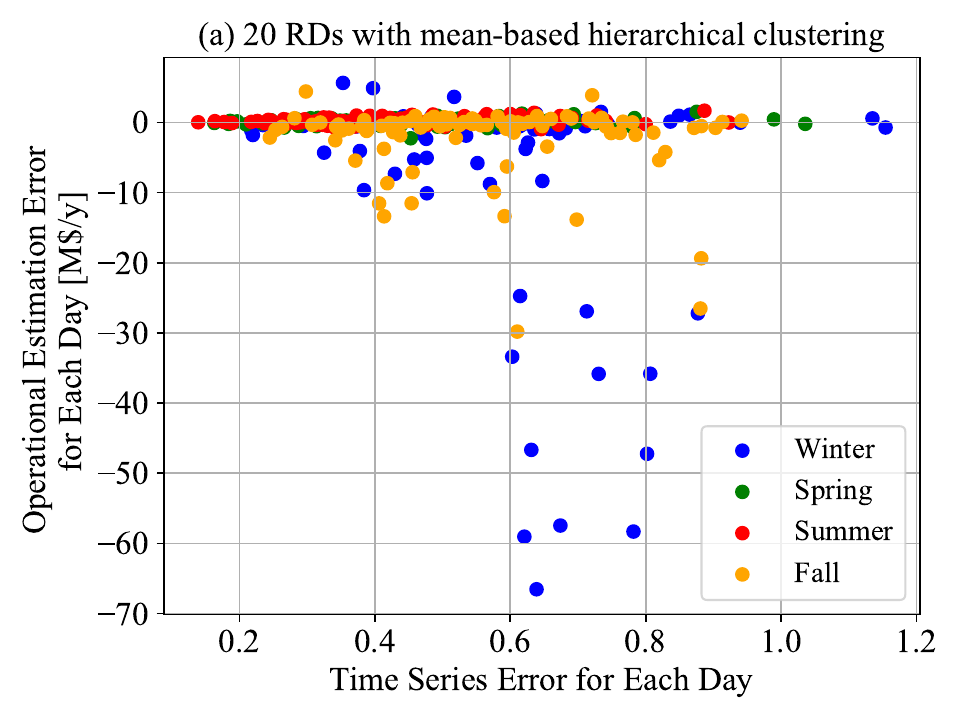}%
        \label{fig:comparison_oc_ts_20}}
        \vspace{-5mm}
    \hfil
        \subfloat{\includegraphics[width=0.9\columnwidth]{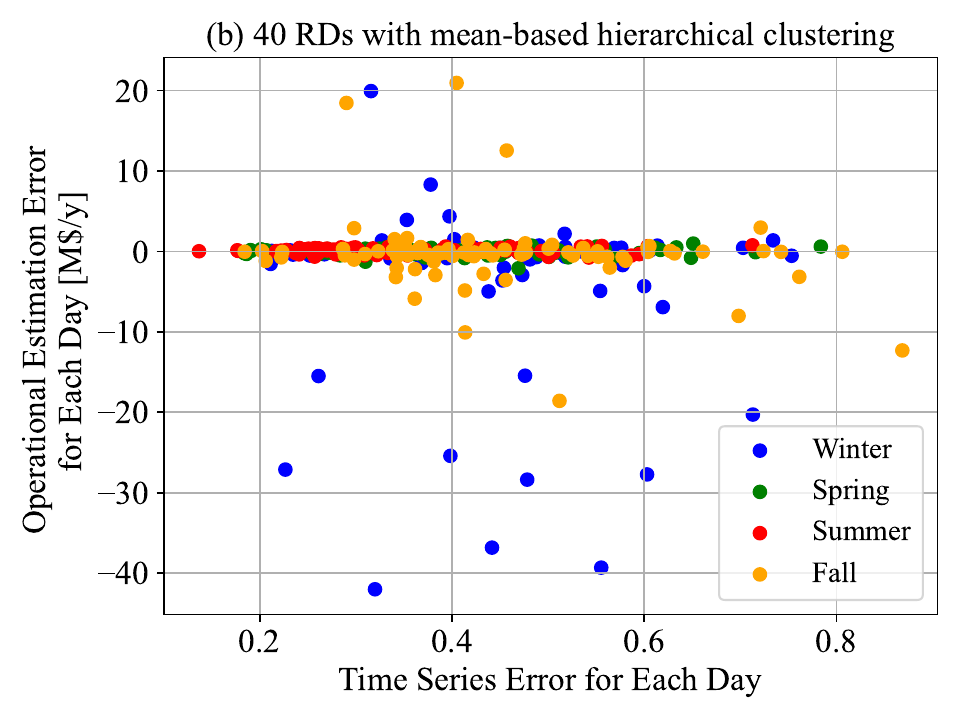}%
        \label{fig:comparison_oc_ts_40}}
    \caption{Distribution of operational estimation error and time-series error across original days, for 20 and 40 RDs.}
    \label{fig:comparison_oc_ts}
\end{figure}

It is noteworthy that these ill-estimated days with high operational estimation errors are primarily located in fall and winter, but these do not show exceptionally high time series errors. In the RD selection, the difference of load and wind factors between days are considered instead of their actual levels. However, the operational cost is not linearly related to the time series error. When the actual load factor is high, the power system operates closer to its operational limits to meet the load requirements. The time series error may either overlook or incorrectly trigger power congestion, resulting in a much higher operational cost error compared to normal cases. This highlights the system-specific, nonlinear sensitivity between scenarios and system operation, and the limits of purely input-based clustering approaches.

\subsection{Operational Estimation Error over Representative Days}
\begin{figure}[bhtp]
    \centering
        \subfloat{\includegraphics[width=0.9\columnwidth]{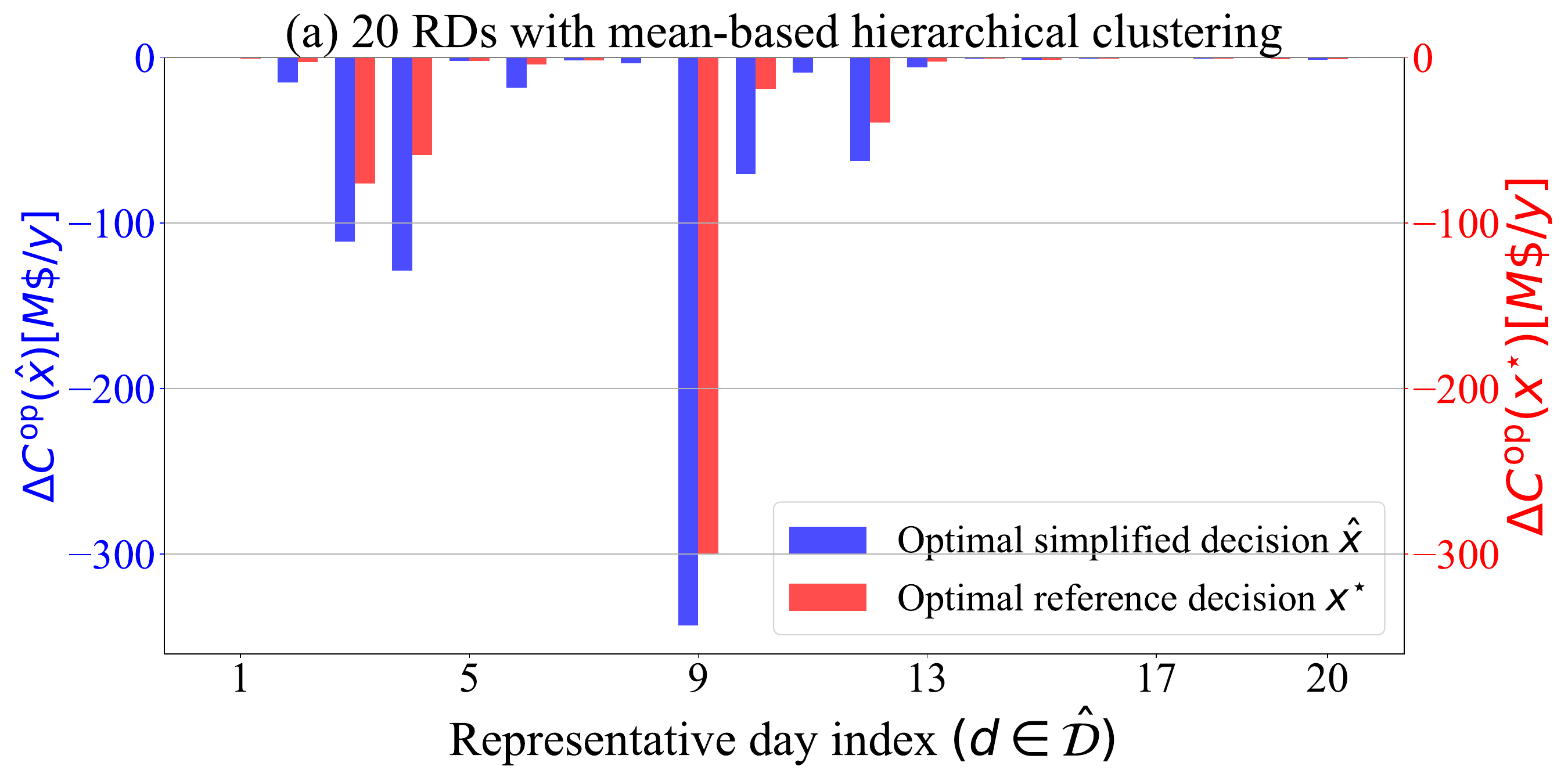}%
        \label{fig:OC_RD_combined_20}}
        \vspace{-3mm}
    \hfil
        \subfloat{\includegraphics[width=0.9\columnwidth]{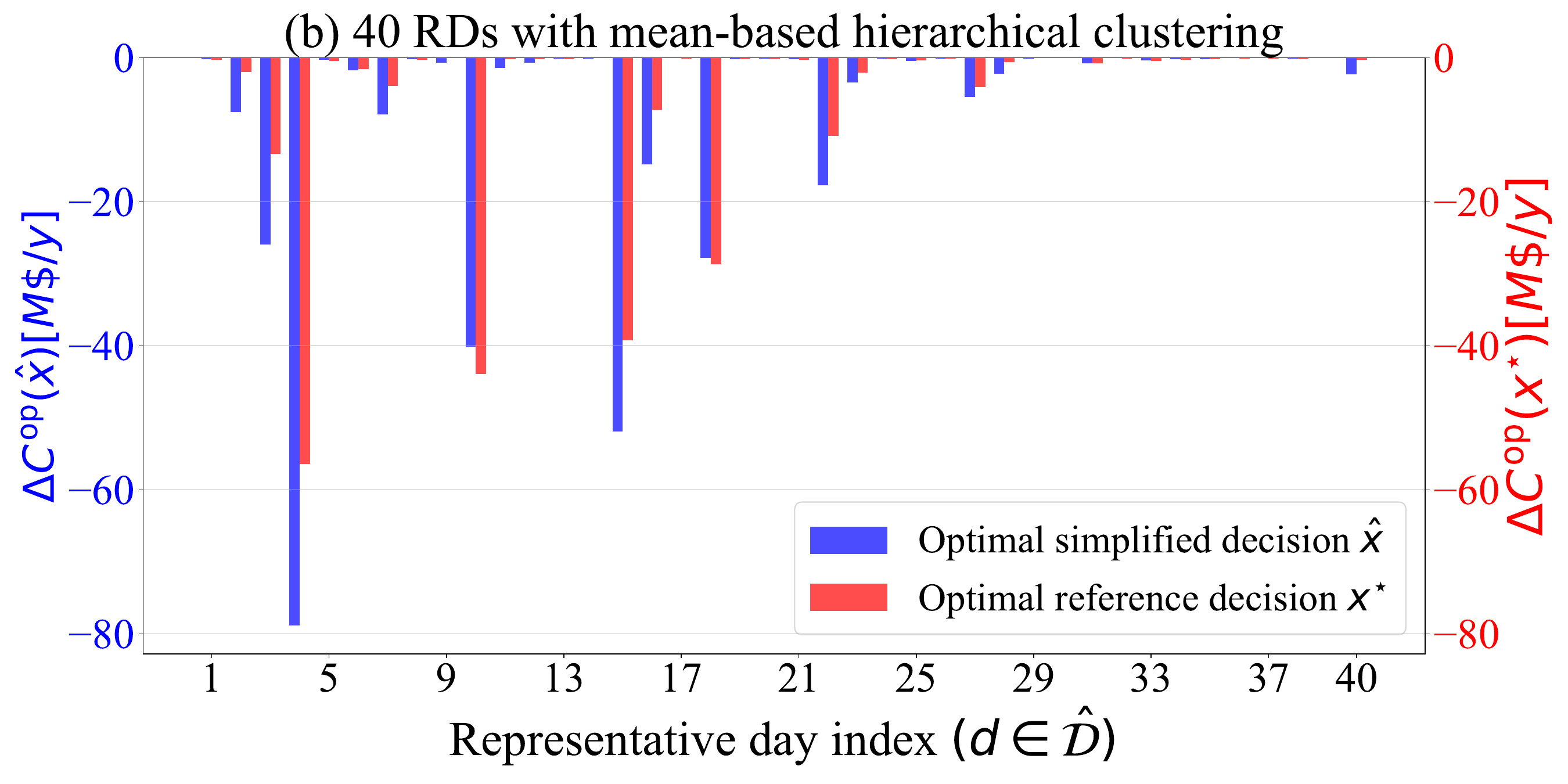}%
        \label{fig:OC_RD_combined_40}}
        \vspace{-3mm}
    \hfil
        \subfloat{\includegraphics[width=0.9\columnwidth]{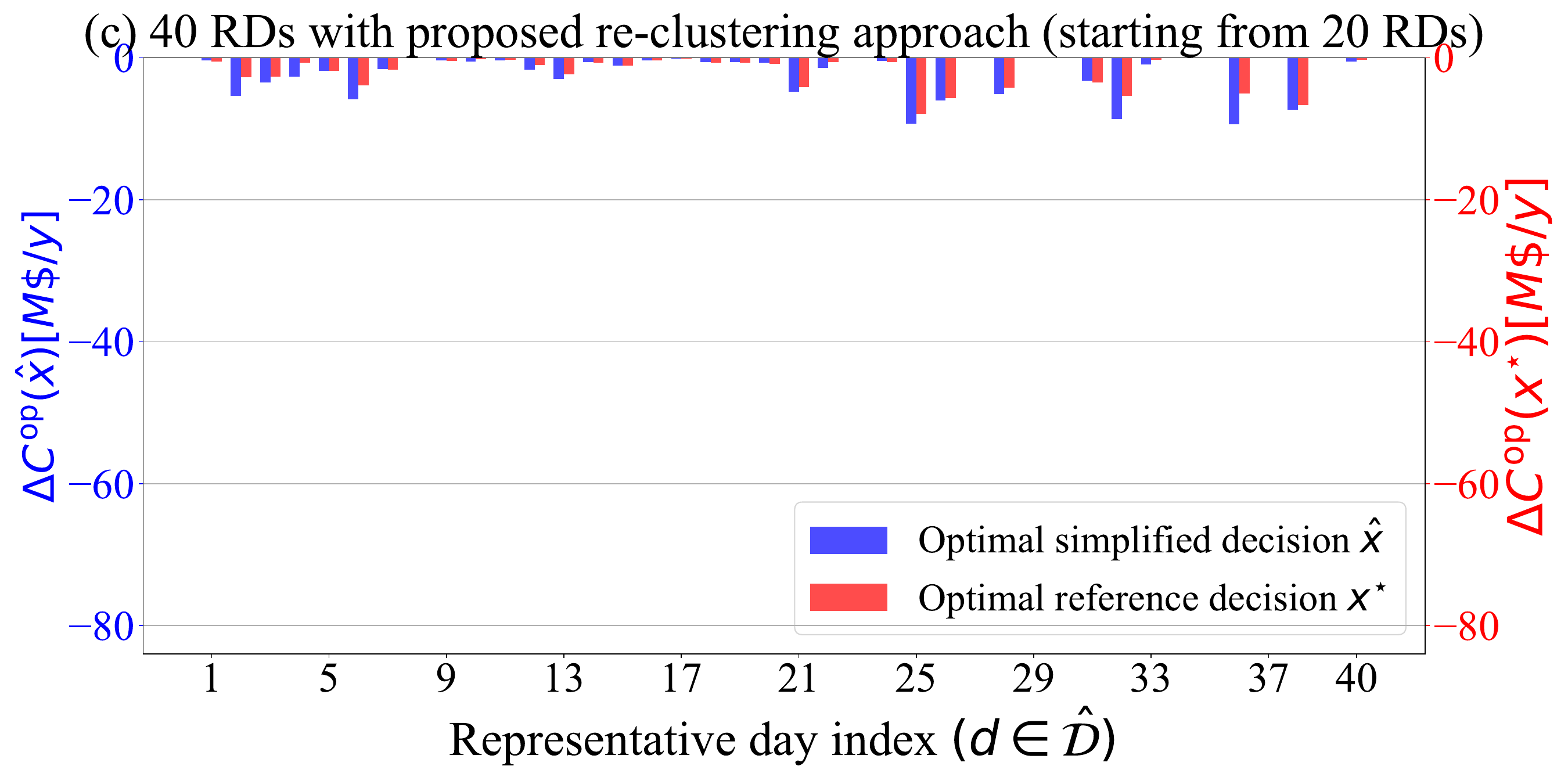}%
        \label{fig:OC_RD_combined_40_fd}}
    \caption{Distribution of operational estimation error across representative days, evaluated for the approximate solution $\hat{x}$ resulting from the use of RDs (blue) and the reference solution $x^\star$ (red).}
    \label{fig:OC_RD_combined}
\end{figure}
We next analyze how the operational estimation error is distributed over RDs. Fig.~\ref{fig:OC_RD_combined_20} and Fig.~\ref{fig:OC_RD_combined_40} show the operational estimation error distribution for 20 and 40 RDs. For each RD, the operational estimation error associated with its cluster of original days is
\begin{equation} \label{eq:RDopcosterror}
    \Delta C_k^{\mathrm{op}}\left(x\right) = 
    C^{\mathrm{op}}\left(x, \{(\hat{w}_k,\hat{T}_k)\}\right) - C^{\mathrm{op}}\left(x, \mathcal{D}_k\right).
\end{equation}
All RDs are mean-based and therefore underestimate the actual operational costs (negative errors), consistent with Proposition~\ref{proposition:Under-estimation}. The distribution is again extremely uneven such that a small number of RDs contributed to most of the total error, especially at small RD counts. Moreover, large difference across $\hat{x}$ and $x^\star$ can be observed also on those problematic RDs, driving both the value of the decision error and the tightness of the bound in Proposition~\ref{proposition:general_bound_decision_error}. Simply increasing the number of RDs does not fully resolve these outliers, demonstrating that clustering based on statistical similarity does not efficiently reduce the operational estimation error.
\section{Feedback Enhancement of Time series Aggregation}
\subsection{Proposed Approach}
The analysis in Section~\ref{sec:Analysis of Operational Estimation Error} demonstrated that the operational estimation error is highly unevenly distributed among RDs, with a small subset contributing to the majority of the total error. By running the full-space operational cost model, these RDs with high operational estimation errors and their linked original days can be located. To tighten the bound using Theorem~\ref{theorem:Bounding_OptimalityGap}, an iterative re-clustering TSA enhancement process is proposed to further cluster the identified poorly performing RDs, while keeping the other well-performing RDs. 

The proposed method, shown in Algorithm~\ref{alg:recluster}, first used standard clustering using a relatively small number of RDs, obtained via mean-based hierarchical clustering. The reduced-scale expansion planning model is then solved to determine the estimated investment decision $\hat{x}$. This decision is evaluated in the full-space operational cost model to compute the operational estimation error for each RD. By analyzing the distribution of these errors, the RDs with the largest weighted errors are identified as the most problematic. The original days linked to these poorly performing RDs are then isolated and subjected to an additional round of clustering, this time with a larger number of RDs to better capture their variability. The newly generated RDs replace the original problematic ones, forming an updated RD set. The refinement continues until a target number of RDs is reached.
\begin{algorithm}[t]
    \caption{Feedback re-clustering of worst-performing RDs}
    \label{alg:recluster}
    \begin{algorithmic}[1]
        \STATE Inputs: initial RD count $N_0$, number of feedback loops $N_{\mathrm{loop}}$, per-loop RD increase $N_{\mathrm{step}}$, number of worst RDs to refine $N_{\mathrm{bad}}$.
        \STATE Initialize $nrd \gets N_0$; build initial RD set $\hat{\mathcal{D}}(nrd)$ by standard clustering.
        \WHILE{$nrd < N_0 + N_{\mathrm{loop}} \cdot N_{\mathrm{step}}$}
            \STATE Solve the reduced-scale planning on $\hat{\mathcal{D}}(nrd)$; obtain $\hat{x}$ and $C^{\mathrm{tot}}(\hat{x},\hat{\mathcal{D}}(nrd))$.
            \STATE Evaluate $\hat{x}$ on $\mathcal{D}^{\mathrm{full}}$ to get $C^{\mathrm{tot}}(\hat{x},\mathcal{D}^{\mathrm{full}})$ using the full-scale operational cost model; compute the operational estimation error for each RD using \eqref{eq:RDopcosterror}: $\{\Delta C^{\mathrm{op}}_k\}_{k=1}^{nrd}$.
            \STATE Identify $N_{\mathrm{bad}}$ RDs with the largest $|\Delta C^{\mathrm{op}}_k|$.
            \STATE Collect original days associated with the poor-performing RDs and re-cluster them into $N_{\mathrm{bad}} + N_{\mathrm{step}}$ clusters.
            \STATE Replace the located poor-performing RDs in $\hat{\mathcal{D}}(nrd)$ with the new RDs to obtain $\hat{\mathcal{D}}(nrd + N_{\mathrm{step}})$.
            \STATE $nrd \gets nrd + N_{\mathrm{step}}$.
    \ENDWHILE
    \STATE Output: $\hat{\mathcal{D}}(nrd)$.
    \end{algorithmic}
\end{algorithm}

\subsection{Performance Analysis}
To evaluate the performance of the proposed re-clustering approach, two feedback variants with mean-based hierarchical clustering were tested: starting from $N_0{=}20$ and from $N_0{=}30$ RDs. In each iteration we refined the worst-performing cluster incrementally ($N_{\mathrm{bad}}=1; N_{\mathrm{step}}=1$). Results were benchmarked against direct mean-based hierarchical clustering for the same final RD count. The full-scale problem was also solved to optimality to determine the decision error.

Fig.~\ref{fig:deltaTC_a101_1} shows (i) the normalized decision error and (ii) the normalized operational estimation error (objective bounding) as functions of the RD count. Both metrics are normalized by the full-resolution optimal total cost $C^{\mathrm{tot}}(x^\star,\mathcal{D}^{\mathrm{full}})$ for better interpretability. Across all settings, increasing the number of RDs improves performance but with diminishing effects. Relative to the baseline, the feedback method shows significant gains, especially at small RD counts where the imbalanced distribution in operational estimation errors is severe. Moreover, we point out that the upper bound is especially reduced, which is beneficial for bounding the impact of the TSA procedure.

The proposed feedback mechanism directly targets at the operational estimation error established in Section~\ref{sec:Error Bounds of Reduced Problems}. Fig.~\ref{fig:OC_RD_combined_40_fd} demonstrates this, showing that by focusing additional representational capacity on the most problematic clusters, the method achieves more balanced error distributions, improving both operational accuracy and the quality of investment decisions. Notably, when baseline improvements begin to drop, inserting a feedback step provides a clear additional boost by mitigating the imbalance in the error distribution. Between the two variants, starting feedback earlier ($N_0{=}20$) achieves a larger improvement for a given number of RDs, as it inserts more steps that prioritize representational capacity for poorly performing clusters. 

In practical use, the benefit of feedback enhancement should be balanced with the additional cost of evaluating the full-scale operational model. In the common case where the investment model has many integer decision variables, computation time increases dramatically with the number of investment options. The 24-bus reference model with four transmission lines and four wind farm candidates required 18,694\,s; adding one more line and two wind farms increased it to 32,196\,s. In contrast, solving the corresponding linear operational model with the full scenario set took 384\,s and 364\,s, respectively. The reduced-scale planning model required 37\,s for 20 RDs, 156\,s for 40, and 360\,s for 60.  Therefore, directing solving the full or even moderately reduced planning problems can become computationally difficult for large-scale systems, whereas, the linear operational models remain feasible, which can be further accelerated with parallel computing. Consequently, performing multiple feedback enhancement steps is worthwhile. This paper focused on the potential improvement of outcomes, which can serve to develop practical iterative procedures. 

\begin{figure}[tbp]
    \centering
    \includegraphics[width=.88\columnwidth]{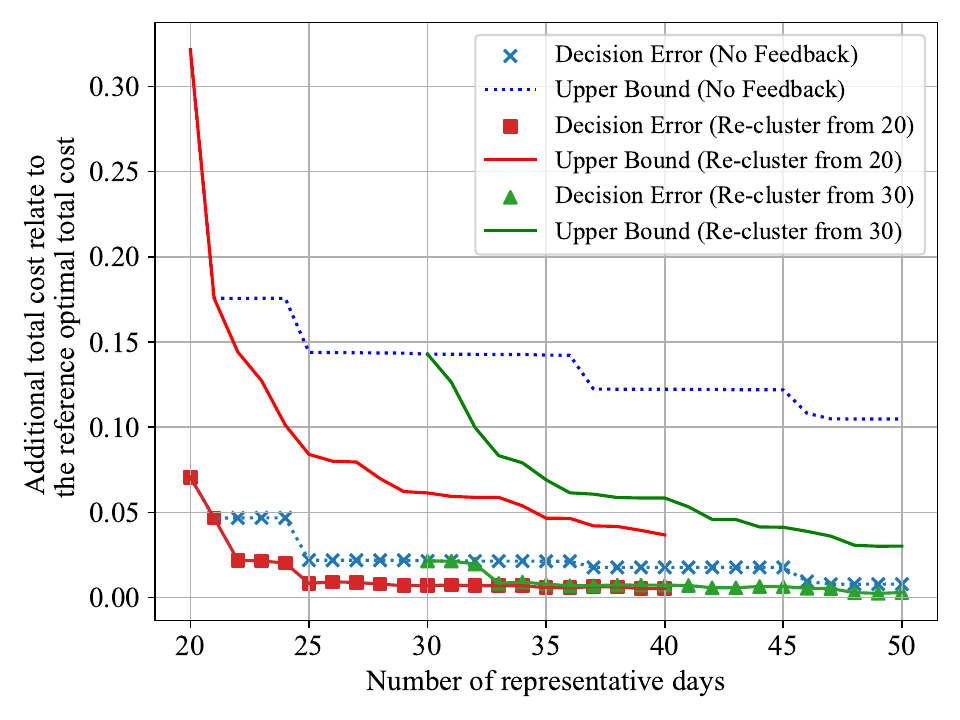}
    \caption{Performance of reduced-scale planning model with different number of RDs using iterative re-clustering (starting from 20 and 30) and benchmarking with the direct mean-based hierarchical clustering method.}
    \label{fig:deltaTC_a101_1}
\end{figure}

\section{Conclusions and Future Work}
This paper investigates the impact of TSA on the expansion planning decisions with RDs, directly targeting the objective function, the investment and operational cost. We showed that, for all TSA approaches, the error in the objective function compared to the reference decision (decision error) is bounded by an expression using only the approximation error of the operational cost. Furthermore, if TSA methods use mean-based clustering and are applied to relatively simple operational models, the operational estimation error for the estimated decision $\hat{x}$ is a practical upper bound for the total decision error. It was observed that a small subset of representative (and original) days contributes disproportionally to the total operational estimation error. A feedback re-clustering mechanism was proposed to specifically increase the representational capacity for the days associated with these RDs. In a case study, this significantly reduced the decision error and tightened the bound relative to standard clustering approaches. 

Future work will further optimize the computational efficiency of the method by identifying re-clustering strategies that minimize the number of sequential optimizations. The framework will also be extended to settings with long-term storage, considering inter-period chronology by using linked RDs (e.g., \cite{kotzur_time_2018}), analyzing the impact of TSA with linked RDs in power system planning.

\renewcommand*{\bibfont}{\footnotesize}
\printbibliography

\end{document}